\documentclass{article}
\usepackage{amsmath}
\usepackage{bbold}
\usepackage{amsfonts}
\usepackage{graphicx}
\usepackage{amsthm}
\usepackage{amssymb}
\usepackage{float}
\newtheorem{thm}{Theorem}
\newtheorem{proposition}{Proposition}
\newtheorem{lemma}{Lemma}
\newtheorem{corollary}{Corollary}
\newtheorem{remark}{Remark}

\usepackage{breqn}
\newtheorem{example}{Example}

\title{Enumeration of $m$-Endomorphisms}

\author{Louis Rubin and Brian Rushton}

\begin{document}
\maketitle

\begin{abstract}An $m$-endomorphism on a free semigroup is an endomorphism that sends every generator to a word of length $\leq m$. Two $m$-endomorphisms are combinatorially equivalent if they are conjugate under an automorphism of the semigroup. In this paper, we specialize an argument of N. G. de Bruijn to produce a formula for the number of combinatorial equivalence classes of $m$-endomorphisms on a rank-$n$ semigroup. From this formula, we derive several little-known integer sequences.
\end{abstract}

\section{Acknowledgments} 
We thank the anonymous referee, whose numerous observations and suggestions led to substantial revision. This research was supported by Temple University's Undergraduate Research Program. 
\section{Introduction}

Let $D$ be a nonempty set of symbols, and let $D^+$ be the set of all finite strings of one or more elements of $D$. That is, $D^+=\{d_1\ldots d_k| k\in \mathbb{N}, d_i\in D\}$. Paired with the operation of string concatenation, $D^+$ forms the \textbf{free semigroup} on $D$. If $d_1,\ldots,d_k\in D$, then we refer to the natural number $k$ as the \textbf{length} of the string $d_1\ldots d_k$. Denote the length of $W \in D^+$ by $\left\vert W \right\vert$.

By a \textbf{semigroup endomorphism} (or, simply, an \textbf{endomorphism}) on $D^+$, we mean a mapping $\phi:D^+\to D^+$ satisfying $\phi(W_1W_2)=\phi(W_1)\phi(W_2)$ for all $W_1,W_2 \in D^+$. Note that if $\phi$ is an endomorphism on $D^+$ and $d_1,\ldots,d_k\in D$, then $\phi(d_1\ldots d_k)=\phi(d_1)\ldots \phi(d_k)$; this shows that an endomorphism on $D^+$ is determined by its action on the elements of $D$. On the other hand, any mapping $f:D\to D^+$ extends uniquely to  the endomorphism $\phi_f:D^+\to D^+$ defined by $\phi_f(d_1\ldots d_k)=f(d_1)\ldots f(d_k)$, and it is straightforward to verify that $\phi_f$ is an automorphism (that is, a bijective endomorphism) precisely when $f$ is a bijection on $D$. 
 
\begin{example} \normalfont Let $D=\{a,b\}$, and let $f:D\to D^+$ be defined by $f(a)=ab$ and $f(b)=a$. Then, for example, 
\begin{align*}
\phi_f(ababa)=f(a)f(b)f(a)f(b)f(a)=abaabaab.
\end{align*}

\end{example}

Let $End(D^+)$ be the collection of all endomorphisms on $D^+$, and let $m\in\mathbb{N}$. Then $\phi\in End(D^+)$ is called an \textbf{\textit{m}-endomorphism} 
if and only if $\left\vert \phi(d) \right\vert \leq m$ for all $d\in D$. Note that the mapping $\phi_f$ from Example 1 is an $m$-endomorphism for all $m\geq 2$. Now let $\Gamma$ be the set of all $m$-endomorphisms on $D^+$. That is,
\begin{align*}
\Gamma &=\{\phi\in End(D^+):\phi(D)\subseteq R\},\\
\end{align*}
where $R=\{W\in D^+: \left\vert W\right\vert \leq m\}$. Consider the set $\Omega$ consisting of all mappings $f:D\to R$. Then we may write 
\begin{align*}
\Gamma &=\{\phi_f:f\in \Omega \}.\\
\end{align*}

We can put $\Gamma$ into one-to-one correspondence with $\Omega$ by sending each $m$-endomorphism to its restriction to $D$. Moreover, if $\left\vert D\right\vert = n\in \mathbb{N}$, then the size of these sets is easily evaluated in view of the fact that $\left\vert R\right\vert= \sum\limits_{i=1}^m n^i$. In particular, if $n>1$, then $\left\vert R\right\vert= \frac{n^{m+1}-n}{n-1}$, and 
\begin{align*}
\left\vert\Gamma \right\vert=\left\vert \Omega \right\vert = \bigg(\frac{n^{m+1}-n}{n-1}\bigg)^n.
\end{align*}

However, in this paper we shall be interested in counting the number of \textit{classes} of $m$-endomorphisms under a particular equivalence relation. To motivate our definition of equivalence on $\Gamma$, we define a relation $\sim$ on $\Omega$ as follows:
\begin{center}
$f_1\sim f_2$ $\iff$ there exists a bijection $g:D\to D$ such that 
$f_2\circ g=\phi_g \circ f_1$.
\end{center}
As an exercise, the reader may wish to verify that $\sim$ satisfies the reflexive, symmetric, and transitive properties required of any equivalence relation. In \S 2.1, however, it will be shown that 
$\sim$ is a specific instance of a well-known equivalence relation induced by a group acting on a nonempty set. 

\begin{example} \normalfont Let $f$ be as in Example 1 (with $D=\{a,b\}$). Consider the bijection $g:D\to D$ defined by $g(a)=b$ and $g(b)=a$. Now let $f_1:D\to D^+$ be given by $f_1(a)=b$ and $f_1(b)=ba$. Then 
\begin{align*}
(f_1\circ g)(a)=f_1(g(a))=f_1(b)= ba= g(a)g(b)= \phi_g(ab)=\phi_g(f(a))&= (\phi_g\circ f)(a)
\end{align*}
and 
\begin{align*}
(f_1\circ g)(b)=f_1(g(b))=f_1(a)= b= g(a)= \phi_g(a)=\phi_g(f(b))&= (\phi_g\circ f)(b),\\
\end{align*}
which shows that $f\sim f_1$.
\end{example}

\begin{remark} \normalfont
Perhaps a more intuitive illustration of $\sim$ is as follows. If we let $f$ and $f_1$ be as in the preceding example, then the respective graphs of $f$ and $f_1$ are $\{(a,ab),(b,a)\}$ and $\{(a,b),(b,ba)\}$. But the graph of $f_1$ can be obtained by applying the bijection $g$ to each element of $D$ that appears in the graph of $f$. In other words, 
\begin{align*}
\{(g(a),g(a)g(b)),(g(b),g(a))\}&=\{(a,b),(b,ba)\}.\\
\end{align*}
Since the graphs of $f$ and $f_1$ are ``the same'' up to a permutation of $a$ and $b$, we wish to consider these mappings equivalent, and $\sim$ provides the desired equivalence relation. 
\end{remark}

Extending $\sim$ to an equivalence relation on $\Gamma$ leads to the following definition. If $f,h\in \Omega$, then  $\phi_f$ is \textbf{combinatorially equivalent} to $\phi_h$ if and only if there exists a bijection $g:D\to D$ such that 
$\phi_h\circ\phi_g=\phi_g\circ \phi_f$. To state precisely the aim of this paper: Given a set of symbols $D$ with $\left\vert D\right\vert =n$, we wish to produce a formula for the number of equivalence classes in $\Gamma$ under the relation of combinatorial equivalence. To this end, we shall specialize an argument of N. G. de Bruijn (namely, that for Theorem 1 in \cite{de1972enumeration}) to produce a formula for the number of classes in $\Omega$ under the relation $\sim$. But it is easy to check that for all $f,h \in \Omega$, $f\sim h$ if and only if $\phi_f$ is combinatorially equivalent to $\phi_h$. Hence, there is a well-defined correspondence given by 
\begin{center}
$[f]\leftrightarrow
[\phi_f]$
\end{center}
between the equivalence classes in $\Omega$ and those in $\Gamma$, and it follows that our formula will also provide the number of $m$-endomorphisms on $D^+$ up to combinatorial equivalence. Moreover, once this formula is obtained, we can fix one of the variables $n,m$ and let the other run through the natural numbers in order to derive integer sequences, many of which appear to be little-known.

\subsection{Group Actions}

For the reader's convenience, we review group actions. The following material (through Proposition 1) is paraphrased from \cite{malik1997fundamentals}. Let $G$ be a group and $S$ a nonempty set. 
A \textbf{left action} of $G$ on $S$ is a function 
\begin{center}
$\cdot: G\times S \to S,$\\ 
$\cdot (g,s)\to g\cdot s$
\end{center}
such that for all $g_1,g_2\in G$ and for all $s\in S$, 
\begin{enumerate}
\item $(g_1g_2)\cdot s = g_1\cdot (g_2\cdot s)$ (where $g_1g_2$ denotes the product of $g_1,g_2$ in $G$), and

\item $e\cdot s = s$ (where $e$ is the identity element of $G$).
\end{enumerate}
A left action induces the well-known equivalence relation $E$ on the set $S$ given by 
\begin{center}
$(a,b)\in E \iff g\cdot a = b$ for some $g\in G$,
\end{center}
for all $a,b\in S$.
We refer to the equivalence classes under this relation as the \textbf{orbits} of $G$ on $S$. The following result (known as ``Burnside's Lemma") gives an expression for the number of these, provided that $G$ and $S$ are finite.
\begin{proposition}\normalfont{\cite{malik1997fundamentals}} Let $S$ be a finite, nonempty set, and suppose there is a left action of a finite group $G$ on $S$. Then the number of orbits of $G$ on $S$ is
\begin{center}
$\frac{1}{\left\vert G\right\vert}\sum\limits_{g\in G}\left\vert \{s\in S: g\cdot s = s\}\right\vert$.
\end{center}
\end{proposition}
Thus, the number of orbits of $G$ on $S$ equals the average number of elements of $S$ that are ``fixed" by an element of $G$. We now show that the relation $\sim$ from \S 2 is a specific instance of the relation $E$ described above. To see this, let $D$ be a finite nonempty set, and let $Sym(D)$ denote the symmetric group on $D$ (i.e., the group of all bijections on $D$). Then $Sym(D)$ acts on the set $\Omega$ according to the rule 
\begin{center}
$g\cdot f = \phi_g \circ f\circ g^{-1}$,
\end{center}
for all $g\in Sym(D), f\in \Omega$. (One can easily verify that $\cdot$ defined in this way is indeed a left action.) Now, for any $f_1,f_2\in \Omega$, we have 
\begin{align*}
f_1\sim f_2 &\iff  f_2\circ g = \phi_g \circ f_1 \text{ for some } g\in Sym(D) \\ &\iff f_2= \phi_g\circ f_1 \circ g^{-1}  \text{ for some } g\in Sym(D) \\  &\iff  g\cdot f_1 = f_2  \text{ for some } g\in Sym(D)\\ &\iff (f_1,f_2)\in E.\\
\end{align*}

It follows that the equivalence classes in $\Omega$ under the relation $\sim$ are just the orbits of $Sym(D)$ on $\Omega$. Enumerating the elements of $Sym(D)$ by $g_1,\ldots,g_{n!}$, we find the number of orbits to be 
\begin{equation}
\frac{1}{n!}\sum\limits_{r=1}^{n!}
\left\vert\{f\in \Omega: f\circ g_r=\phi_{g_r}\circ f\}\right\vert\\.
\end{equation}

For any permutation $g$ of a finite set, and for each natural number $j$, let $c(g,j)$ denote the number of cycles of length \footnote{There should be no confusion between the notions of `string length' and `cycle length'.} $j$ occurring in the cycle decomposition of $g$. (This notation comes from \cite{de1972enumeration}.) The quantities $c(g,j)$ will play a role in the evaluation of $\left\vert\{f\in \Omega: f\circ g_r=\phi_{g_r}\circ f\}\right\vert$, which occurs in the next section. Our evaluation is a modification of de Bruijn's counting argument in \S 5.12 of \cite{de1964polya}. 

\section{Main Results}

We now produce a formula for the number of equivalence classes in $\Omega$ under the relation $\sim$. Let $D$ be a finite set, and suppose that $g\in Sym(D)$ is the product of disjoint cycles of lengths $k_1,k_2,\ldots,k_{\ell}$, where $k_1\leq k_2\leq \ldots \leq k_{\ell}$. Then the sequence $k_1,k_2,\ldots,k_{\ell}$ is called the \textbf{cycle type} of $g$. For example, if $D=\{a,b,c,d,e\}$, then the permutation $g=(a)(b,c)(d,e)$ has cycle type 1,2,2. The following lemma will be useful.  
\begin{lemma}
Let $D$ be a finite set, and let  $g\in Sym(D)$ have cycle type $k_1,k_2,\ldots,k_{\ell}$. For each $1\leq i\leq \ell$, select a single $d_i\in D$ from the cycle corresponding to $k_i$. (Thus, $k_i$ is the smallest natural number such that $g^{k_i}(d_i)=d_i$.) Now suppose that $f\in \Omega$. Then
$f\circ g=\phi_g \circ f$ if and only if for each $1\leq i\leq \ell$, the following holds:
\begin{enumerate}
\item $(f\circ g^j)(d_i)=(\phi_g^j\circ f)(d_i)$ for all $j\in \mathbb{N}.$
\item $f(d_i)$ is of the form $d'_1\ldots d'_{k\leq m}$, where $d'_1,\ldots, d'_k\in D$ each belong to a cycle in $g$ whose length divides $k_i$. 
\end{enumerate} 
\end{lemma}

\begin{proof}
First assume that $f\circ g=\phi_g\circ f$. Then condition 1 follows from an inductive argument. But $f(d_i)=f(g^{k_i}(d_i))=\phi_g^{k_i}(f(d_i))$. Write $f(d_i)=d'_1\ldots d'_k$, where $d'_1,\ldots, d'_k\in D$ and $k\leq m$. Then 
\begin{center}
$d'_1\ldots d'_k=\phi_g^{k_i}(d'_1\ldots d'_k)=g^{k_i}(d'_1)\ldots g^{k_i}(d'_k).$
\end{center}
In particular, for each $1\leq t\leq k$, we have $d'_t=g^{k_i}(d'_t)$. This implies that 
\begin{center}
$\Big(d'_t,g(d'_t),g^2(d'_t),\ldots,g^{k_i-1}(d'_t)\Big)$
\end{center}
is a cycle whose length divides $k_i$. The conclusion follows. 

Conversely, suppose that condition 1 holds. (Condition 2 is superfluous here.) Let $d\in D$. Then there exist $i,j\in \mathbb{N}$ such that $d=g^j(d_i)$. Now, 
\begin{align*}
f(g(d))&=f(g(g^j(d_i)))\\&=f(g^{1+j}(d_i))\\&=\phi_g^{1+j}(f(d_i))\\&=\phi_g(\phi_g^j(f(d_i)))\\&=\phi_g(f(g^j(d_i)))\\&=\phi_g(f(d)).\\
\end{align*}
Therefore, $f\circ g=\phi_g\circ f$, so the proof is complete. 
\end{proof}

Once again, suppose that $\left\vert D\right\vert =n$, and label the elements of $Sym(D)$ by $g_1,\ldots,g_{n!}$. For each $1\leq r\leq n!$, we can find the number of $f\in \Omega$ satisfying 
\begin{equation}
f\circ g_r=\phi_{g_r}\circ f.
\end{equation}
Suppose that $g_r$ has cycle type $k_{r1},k_{r2},\ldots,k_{r\ell_r}$. For each $1\leq i\leq \ell_r$, select a single element $d_{ri}\in D$ from the cycle corresponding to $k_{ri}$. Then Lemma 1 implies that any $f\in \Omega$ satisfying (2) is determined by its values on each $d_{ri}$. Hence, to find the number of $f$ satisfying (2), we need only count the number of possible images of $d_{ri}$ under such an $f$, and then take the product over all $i$. But the $m$ or fewer elements of $D$ comprising the string $f(d_{ri})$ must each belong to a cycle in the decomposition of $g_r$ whose length divides $k_{ri}$. For each  $1\leq k\leq m$, there are 
\begin{center}
$\Big(\sum\limits_{j|k_{ri}}jc(g_r,j)\Big)^k$
\end{center}
choices of $f(d_{ri})$ such that 
$\left\vert f(d_{ri})\right\vert=k$. Hence, there are 
\begin{center}
$\sum\limits_{k=1}^m\Big(\sum\limits_{j|k_{ri}}jc(g_r,j)\Big)^k$
\end{center}
total choices of $f(d_{ri})$. Taking the product over all $i$, it follows that the number of $f$ satisfying (2) is 
\begin{equation}
\prod\limits_{i=1}^{\ell_r}\bigg(\sum\limits_{k=1}^m\Big(\sum\limits_{j|k_{ri}}jc(g_r,j)\Big)^k\bigg).
\end{equation}

Thus, we've evaluated $\left\vert\{f\in \Omega: f\circ g_r=\phi_{g_r}\circ f\}\right\vert$, and putting together (1) and (3) gives an expression for the number of equivalence classes in $\Omega$ under the relation $\sim$. Recalling that these classes are in one-to-one correspondence with the classes in $\Gamma$ under the relation of combinatorial equivalence, we obtain our main result:
\begin{thm}
If $\left\vert D\right\vert =n$, then the number of $m$-endomorphisms on $D^+$, up to combinatorial equivalence, is the value of 
\begin{equation}
\frac{1}{n!}\sum\limits_{r=1}^{n!}
\Bigg(\prod\limits_{i=1}^{\ell_r}\bigg(\sum\limits_{k=1}^m\Big(\sum\limits_{j|k_{ri}}jc(g_r,j)\Big)^k\bigg)\Bigg),
\end{equation}
where $g_1,\ldots,g_{n!}$ are the elements of $Sym(D)$, and $k_{r1},\ldots,k_{r\ell_r}$ is the cycle type of $g_r$.  
\end{thm}

\begin{example} \normalfont Let $D=\{a,b\}$. We find the number of classes of 1-endomorphisms on $D^+$. The elements of $Sym(D)$ (in cycle notation) are $g_1=(a)(b)$ and $g_2=(a,b)$. Evidently, $c(g_1,1)=2$, $c(g_2,1)=0$, and $c(g_2,2)=1$. Using Theorem 1, there are 
\begin{align*}
\frac{1}{2}\bigg(c(g_1,1)^2+2c(g_2,2)\bigg)&=\frac{1}{2}(2^2+2)\\&=3\\
\end{align*}
classes of 1-endomorphisms on $D^+$. These are given by
\begin{align*}
\Bigg\{\begin{tabular}{ l c r }
    $a$ & $\to$ & $a$ \\ 
    $b$ & $\to$ & $b$ \\
\end{tabular}\Bigg\},\text{ } \Bigg\{\begin{tabular}{ l c r }
    $a$ & $\to$ & $b$ \\ 
    $b$ & $\to$ & $a$ \\
\end{tabular}\Bigg\}, \text{ and }\Bigg\{\begin{tabular}{ l c r }
    $a$ & $\to$ & $a$ \\ 
    $b$ & $\to$ & $a$ \\
\end{tabular}\equiv\begin{tabular}{ l c r }
    $a$ & $\to$ & $b$ \\ 
    $b$ & $\to$ & $b$ \\
\end{tabular}\Bigg\}. 
\end{align*}

\end{example}

We can extend the result of Example 3 by fixing $n=2$ and letting $m$ be arbitrary. From (4), we find that the number of classes $m$-endomorphisms on $D^+$, where $\left\vert D\right\vert=2$, is \begin{align*}
\frac{1}{2}\Big((2^{m+1}-2)^2+(2^{m+1}-2)\Big).\\
\end{align*}
Running $m$ through the natural numbers, we obtain values $3,21,105,465,1953,\ldots$. This is the sequence A134057 in the On-line Encyclopedia of Integers. (See \cite{OEIS}.) However, for $n=3$, the number of classes of $m$-endomorphisms becomes 
\begin{align*}
\frac{1}{6}\Bigg(\bigg(\frac{3^{m+1}-3}{2}\bigg)^3 + 3m\bigg(\frac{3^{m+1}-3}{2}\bigg) + 2\bigg(\frac{3^{m+1}-3}{2}\bigg)\Bigg).\\
\end{align*}
Letting $m=1,2,3,4,\ldots$ gives values $7, 304, 9958, 288280,\ldots$. This sequence appears to be little-known, and has been submitted by the authors to the OEIS. 

\subsection{An Alternative Formulation of Theorem 1}

We now present a slight rewording of Theorem 1. In order to compute the number of equivalence classes of $m$-endomorphisms (where $\left\vert D\right\vert =n$), we need not, in practice, consider each element of $Sym(D)$ individually. Rather, we need only consider the cycle types of these permutations. The following well-known result gives the number of permutations in $Sym(D)$ of a given cycle type. 

\begin{proposition}\normalfont{\cite{Dummitt}} Let $\left\vert D\right\vert =n$, and let $g\in Sym(D)$. Suppose that $m_1,m_2,\ldots,m_s$ are the \textit{distinct} integers appearing in the cycle type of $g$. For each $j\in \{1,2,\ldots,s\}$, abbreviate $c_j=c(g,m_j)$. Let $C_g$ be the set of all permutations in $Sym(D)$ whose cycle type is that of $g$. Then  
\begin{equation}
\left\vert C_g\right\vert =\frac{n!}{\prod\limits_{j=1}^sc_j!m_j^{c_j}}.
\end{equation}
\end{proposition}

For convenience, we shall say that $g\in Sym(D)$ \textbf{fixes} the mapping $f\in \Omega$ if and only if $f\circ g=\phi_g\circ f$. Now, two bijections in $Sym(D)$ with the same cycle type must fix the same number of $f\in \Omega$. Therefore, in order to derive an expression for the number of classes of $m$-endomorphisms on $D^+$, we can select a single representative in $Sym(D)$ of each possible cycle type, then determine the number of $f\in \Omega$ fixed by each representative using expression (3), multiply this number by the corresponding value of (5), and then sum up over all of our representatives and divide by $n!$. But the cycle types in $Sym(D)$ are precisely the \textbf{integer partitions} of $n$, namely, the nondecreasing sequences of natural numbers whose sum is $n$. If $p(n)$ denotes the number of integer partitions of $n$, then we may restate Theorem 1 as follows. 
\begin{corollary}
Let $\left\vert D\right\vert=n$, and suppose that $g_1,\ldots,g_{p(n)}\in Sym(D)$ have distinct cycle types. Then the number of $m$-endomorphisms on $D^+$, up to combinatorial equivalence, is the value of 
\begin{equation}
\frac{1}{n!}\sum\limits_{r=1}^{p(n)}
\Bigg(\left\vert C_{g_r}\right\vert\prod\limits_{i=1}^{\ell_r}\bigg(\sum\limits_{k=1}^m\Big(\sum\limits_{j|k_{ri}}jc(g_r,j)\Big)^k\bigg)\Bigg),
\end{equation}
where $k_{r1},\ldots,k_{r\ell_r}$ is the cycle type of $g_r$, and $C_{g_r}$ is as in Proposition 2.
\end{corollary}

\begin{example}\normalfont To illustrate Corollary 1, we find the number of classes of $m$-endomorphisms when $\left\vert D\right\vert =4$. Let $D=\{a,b,c,d\}$. As previously mentioned, the 
cycle types in $Sym(D)$ are the integer partitions of 4. There are five such partitions:
\begin{align*}
4 &= 1+1+1+1\\ &=1+1+2\\&= 2+2\\&= 1+3\\ &= 4.\\
\end{align*}
Hence, the bijections 
\begin{center}
$g_1=(a)(b)(c)(d)$, $g_2=(a)(b)(c,d)$, $g_3=(a,b)(c,d)$, $g_4=(a)(b,c,d)$, and $g_5=(a,b,c,d)$
\end{center}
encompass all possible cycle types in $Sym(D)$. Direct calculation using (5) yields 
\begin{center}
$\left\vert C_{g_1}\right\vert = 1$, $\left\vert C_{g_2}\right\vert = 6$, $\left\vert C_{g_3}\right\vert = 3$, $\left\vert C_{g_4}\right\vert = 8$, and $\left\vert C_{g_5}\right\vert = 6$. 
\end{center}
Thus, by Corollary 1, the number of classes of $m$-endomorphisms when $n=4$ is
\begin{center}
$\frac{1}{24}\Bigg(\bigg(\frac{4^{m+1}-4}{3}\bigg)^4+ 6\Big(2^{m+1}-2\Big)^2\bigg(\frac{4^{m+1}-4}{3}\bigg)+3\bigg(\frac{4^{m+1}-4}{3}\bigg)^2 $
\end{center}
\begin{center}
$+ 8m\bigg(\frac{4^{m+1}-4}{3}\bigg)+6\bigg(\frac{4^{m+1}-4}{3}\bigg)\Bigg).$ 
\end{center}
\end{example}

Proceeding along the lines of Example 4, we find that there are
\begin{center}
$\frac{1}{120}\Bigg(\bigg(\frac{5^{m+1}-5}{4}\bigg)^5+10\bigg(\frac{3^{m+1}-3}{2}\bigg)^3\bigg(\frac{5^{m+1}-5}{4}\bigg)+15m\bigg(\frac{5^{m+1}-5}{4}\bigg)^2$
\end{center}
\begin{center}$+ 20\big(2^{m+1}-2\big)^2\bigg(\frac{5^{m+1}-5}{4}\bigg)+20\big(2^{m+1}-2\big)\bigg(\frac{3^{m+1}-3}{2}\bigg)$
\end{center}
\begin{center}$
+30m\bigg(\frac{5^{m+1}-5}{4}\bigg)+24\bigg(\frac{5^{m+1}-5}{4}\bigg)\Bigg)$ 
\end{center}
classes of $m$-endomorphisms when $n=5$,
and
\begin{center}
$\frac{1}{720}\Bigg(\bigg(\frac{6^{m+1}-6}{5}\bigg)^6+15\bigg(\frac{4^{m+1}-4}{3}\bigg)^4\bigg(\frac{6^{m+1}-6}{5}\bigg)+45\big(2^{m+1}-2\big)^2\bigg(\frac{6^{m+1}-6}{5}\bigg)^2$
\end{center}
\begin{center}$
+15\bigg(\frac{6^{m+1}-6}{5}\bigg)^3+40\bigg(\frac{3^{m+1}-3}{2}\bigg)^3
\bigg(\frac{6^{m+1}-6}{5}\bigg)
+120m\bigg(\frac{3^{m+1}-3}{2}\bigg)\bigg(\frac{4^{m+1}-4}{3}\bigg)$
\end{center}
\begin{center}$
+40\bigg(\frac{6^{m+1}-6}{5}\bigg)^2+90\big(2^{m+1}-2\big)^2\bigg(\frac{6^{m+1}-6}{5}\bigg)+90\big(2^{m+1}-2\big)\bigg(\frac{6^{m+1}-6}{5}\bigg)$
\end{center}
\begin{center}$
+144m\bigg(\frac{6^{m+1}-6}{5}\bigg)+120\bigg(\frac{6^{m+1}-6}{5}\bigg)\Bigg)$
\end{center}
classes of $m$-endomorphisms when $n=6$. Letting $m$ run through $\mathbb{N}$ in these cases, we again obtain sequences that are not well-known. The following tables display the values of (6) for $n,m\leq 6$. 

\begin{table}[H]
\begin{tabular}{|l|l|l|l|l|l|l|}
\hline
  & $n=1$  & $n=2$  & $n=3$  & $n=4$  & $n=5$   \\ \hline
$m=1$ & 1 & 3 & 7 & 19 & 47 \\ \hline
$m=2$ & 2 & 21 & 304 & 6,915 & 207,258 \\ \hline
$m=3$ & 3 & 105 & 9,958 & 2,079,567 & 746,331,322 \\ \hline
$m=4$ & 4 & 465 & 288,280 & 556,898,155 & 2,406,091,382,736 \\ \hline
$m=5$ & 5 & 1,953 & 7,973,053 & 144,228,436,231 & 7,567,019,254,708,782\\ \hline
$m=6$ & 6 & 8,001 & 217,032,088 & 37,030,504,349,475 & 23,677,181,825,841,420,408 \\ \hline

\end{tabular}
\end{table}

\begin{table}[H]
\begin{tabular}{|l|l|l|l|l|l|l|}
\hline
  & $n=6$  \\ \hline
$m=1$ & 130  \\ \hline
$m=2$ & 7,773,622 \\ \hline
$m=3$ & 409,893,967,167 \\ \hline
$m=4$ & 19,560,646,482,079,624 \\ \hline
$m=5$ & 916,131,223,607,107,471,135 \\ \hline
$m=6$ & 42,770,482,829,102,570,213,645,988\\ \hline
\end{tabular}
\end{table}
\begin{remark}\normalfont The sequence $1, 3, 7, 19, 47, 130, \ldots$ is sequence A001372 in the OEIS.
\end{remark}

\section{Two Natural Variations}

In this section, we highlight two natural variations of Corollary 1. First, we restrict our attention to endomorphisms on $D^+$ that send each element of $D$ to a string of length exactly $m$. We then consider $m$-endomorphisms of the so-called free monoid, which contains the empty string. Expressions analogous to those in \S 3 are derived in each case. 

\subsection{\textit{m}-Uniform Endomorphisms}

Fix $n,m\in \mathbb{N}$, and suppose that $\left\vert D\right\vert=n$. Then $\phi \in End(D^+)$ is called an \textbf{\textit{m}-uniform endomorphism} if and only if $\left\vert \phi(d)\right\vert=m$ for each $d\in D$. In this section, we produce a formula for the number of $m$-uniform endomorphisms on $D^+$ up to combinatorial equivalence. To begin, let $g_1,\ldots,g_{p(n)}\in Sym(D)$ have distinct cycle types. We now put $R=\{W\in D^+:\left\vert W\right\vert =m\}$ and take $\Omega$ to be the set of all mappings of $D$ into $R$. For each $1\leq r\leq p(n)$, we ask for the number of $f\in \Omega$ satisfying 
\begin{center}
$f\circ g_r=\phi_{g_r}\circ f$.
\end{center}
Once again, if $g_r$ has cycle type $k_{r1},\ldots,k_{r\ell_r}$, then for each $1\leq i\leq \ell_r$ we select an element $d_{ri}$ from the cycle corresponding to $k_{ri}$, and count the number of possible values of $f(d_{ri})$. In this case, we must have $\left\vert f(d_{ri})\right\vert =m$, where the elements of $D$ comprising the string $f(d_{ri})$ each belong to a cycle whose length divides $k_{ri}$. Hence, there are
\begin{center}
$\Big(\sum\limits_{j|k_{ri}}jc(g_r,j)\Big)^m$
\end{center}
choices of $f(d_{ri})$, and multiplying over all $i$ yields 
\begin{center}
$\prod\limits_{i=1}^{\ell_r}\Big(\sum\limits_{j|k_{ri}}jc(g_r,j)\Big)^m$
\end{center}
as the value of $\left\vert\{f\in \Omega: f\circ g_r=\phi_{g_r}\circ f\}\right\vert$. Noting that permutations in $Sym(D)$ of the same cycle type fix the same number of $f\in \Omega$,  we multiply by $\left\vert C_{g_r}\right\vert $, sum with respect to $r$, and divide by $n!$ to obtain the following. 
\begin{corollary}
If $\left\vert D\right\vert=n$ and  $g_1,\ldots,g_{p(n)}\in Sym(D)$ have distinct cycle types, then the number of $m$-uniform endomorphisms on $D^+$, up to combinatorial equivalence, is the value of 
\begin{equation}
\frac{1}{n!}\sum\limits_{r=1}^{p(n)}
\Bigg(\left\vert C_{g_r}\right\vert\prod\limits_{i=1}^{\ell_r}\Big(\sum\limits_{j|k_{ri}}jc(g_r,j)\Big)^m\Bigg),
\end{equation}
where $k_{r1},\ldots,k_{r\ell_r}$ is the cycle type of $g_r$, and $C_{g_r}$ is as in Proposition 2.
\end{corollary}

When $n=2$, the number of $m$-uniform endomorphisms on $D^+$, up to combinatorial equivalence, is
\begin{center}
$\frac{1}{2}\big(2^{2m}+2^m\big)$.
\end{center}
Letting $m=1,2,3,4,\ldots$ gives 
values $3,10,36,136,\ldots$. This is the sequence A007582 from the OEIS. Moreover, when $n=3$ there are 
\begin{center}
$\frac{1}{6}\big(3^{3m}+3\cdot 3^m+2\cdot 3^m\big)$
\end{center}
classes of $m$-uniform endomorphisms, and letting $m$ run through $\mathbb{N}$ gives the sequence $7,129,3303,88641,\ldots$, which is not well-known. Continuing, the expressions when $n=4,5,6$ are 
\begin{center}
$\frac{1}{24}\big(4^{4m}+6\cdot 2^{2m}\cdot 4^{m}+3\cdot 4^{2m}+8\cdot 4^m + 6\cdot 4^m\big),$ 
\end{center}
\begin{center}
$\frac{1}{120}\big(5^{5m}+10\cdot 3^{3m}\cdot 5^m + 15 \cdot 5^{2m} + 20\cdot 2^{2m}\cdot 5^m + 20\cdot 2^m\cdot 3^m + 30 \cdot 5^m + 24\cdot 5^m\big),$
\end{center}
and
\begin{center}
$\frac{1}{720}\big(6^{6m}+15\cdot 4^{4m}\cdot 6^m + 45\cdot 2^{2m}\cdot 6^{2m} + 15\cdot 6^{3m}+ 40 \cdot 3^{3m}\cdot 6^m$ 
\end{center}
\begin{center}$
+120\cdot 3^m\cdot 4^m + 40\cdot 6^{2m} + 90\cdot 2^{2m} \cdot 6^m + 90\cdot 2^m \cdot 6^m + 144\cdot 6^m + 120 \cdot 6^m\big),$
\end{center}
respectively. The following tables display the values of (7) for $n,m\leq 6$.

\begin{table}[h]
\begin{tabular}{|l|l|l|l|l|l|l|}
\hline
  & $n=1$  & $n=2$  & $n=3$  & $n=4$  & $n=5$   \\ \hline
$m=1$ & 1 & 3 & 7 & 19 & 47 \\ \hline
$m=2$ & 1 & 10 & 129 & 2,836 & 83,061 \\ \hline
$m=3$ & 1 & 36 & 3,303 & 700,624 & 254,521,561 \\ \hline
$m=4$ & 1 & 136 & 88,641 & 178,981,696 & 794,756,352,216 \\ \hline
$m=5$ & 1 & 528 & 7,973,053 & 45,813,378,304 & 2,483,530,604,092,546\\ \hline
$m=6$ & 1 & 2,080  & 64,570,689 &  11,728,130,323,456 & 7,761,021,959,623,948,401 \\ \hline
\end{tabular}
\end{table}

\begin{table}[H]
\begin{tabular}{|l|l|l|l|l|l|l|}
\hline
  & $n=6$  \\ \hline
$m=1$ & 130  \\ \hline
$m=2$ & 3,076,386 \\ \hline
$m=3$ & 141,131,630,530 \\ \hline
$m=4$ & 6,581,201,266,858,896 \\ \hline
$m=5$ & 307,047,288,863,992,988,160 \\ \hline
$m=6$ & 14,325,590,271,500,876,382,987,456 \\ \hline
\end{tabular}
\end{table}

\subsection{The Free Monoid}

If we adjoin the unique string of length 0 (denoted by $\epsilon$) to the set $D^+$, then we form
the set $D^*$. Paired with the operation of string concatenation, $D^*$ forms the \textbf{free monoid} on $D$. We refer to $\epsilon$ as the \textbf{empty string}, and it serves as the identity element in $D^*$. That is, for any $W\in D^*$, 
\begin{center}
$W\epsilon=W=\epsilon W$. 
\end{center}
We define an endomorphism on $D^*$ to be a mapping $\phi:D^*\to D^*$ such that $\phi(W_1W_2)=\phi(W_1)\phi(W_2)$ for all $W_1,W_2\in D^*$. 

\begin{remark}\normalfont
Note that if $\phi$ is an endomorphism on $D^*$, then $\phi(\epsilon)=\epsilon$. This follows since
for any $W\in D^*$, we have
\begin{center}
$\phi(W)=\phi(\epsilon W)=\phi(\epsilon)\phi(W)$,
\end{center}
which implies that $\phi(\epsilon)$ has length 0.
\end{remark}

Now, an $m$-endomorphism on $D^*$ is an endomorphism such that $\left\vert \phi(d)\right\vert \leq m$ for all $d\in D$. Thus, an $m$-endomorphism on $D^*$ can map elements of $D$ to $\epsilon$. To determine the number of $m$-endomorphisms on $D^*$ up to combinatorial equivalence, we put $R=\{W\in D^*:\left\vert W\right\vert \leq m\}$, and for each $g\in Sym(D)$, we ask for the number of 
$f:D\to R$ that are fixed by $g$. Again, it suffices to count the number of possible images under such an $f$ of a single $d\in D$ from each cycle in the decomposition of $g$, and then multiply over all the cycles. But there is now one additional possible value of $f(d)$: the empty string. Hence, if $d$ belongs to a cycle of length $k_i$, then we have
\begin{center}
$1+\sum\limits_{k=1}^m\Big(\sum\limits_{j|k_i}jc(g_r,j)\Big)^k=\sum\limits_{k=0}^m\Big(\sum\limits_{j|k_i}jc(g_r,j)\Big)^k$
\end{center}
choices of $f(d)$. From this observation, we deduce the following. 
\begin{corollary}
Let $\left\vert D\right\vert=n$, and suppose that $g_1,\ldots,g_{p(n)}\in Sym(D)$ have distinct cycle types. Then the number of $m$-endomorphisms on $D^*$, up to combinatorial equivalence, is the value of 
\begin{equation}
\frac{1}{n!}\sum\limits_{r=1}^{p(n)}
\Bigg(\left\vert C_{g_r}\right\vert\prod\limits_{i=1}^{\ell_r}\bigg(\sum\limits_{k=0}^m\Big(\sum\limits_{j|k_{ri}}jc(g_r,j)\Big)^k\bigg)\Bigg),
\end{equation}
where $k_{r1},\ldots,k_{r\ell_r}$ is the cycle type of $g_r$, and $C_{g_r}$ is as in Proposition 2.
\end{corollary}

When $n=2$, the number of $m$-endomorphisms on $D^*$, up to combinatorial equivalence, is
\begin{align*}
\frac{1}{2}\Big((2^{m+1}-1)^2+(2^{m+1}-1)\Big).\\ 
\end{align*}
This is sequence A006516 from the OEIS. The corresponding expressions for $n=3,4,5,6$ are 
\begin{align*}
\frac{1}{6}\Bigg(\bigg(\frac{3^{m+1}-1}{2}\bigg)^3 + 3(m+1)\bigg(\frac{3^{m+1}-1}{2}\bigg) + 2\bigg(\frac{3^{m+1}-1}{2}\bigg)\Bigg),\\
\end{align*}
\begin{center}
$\frac{1}{24}\Bigg(\bigg(\frac{4^{m+1}-1}{3}\bigg)^4+ 6\Big(2^{m+1}-1\Big)^2\bigg(\frac{4^{m+1}-1}{3}\bigg)+3\bigg(\frac{4^{m+1}-1}{3}\bigg)^2$
\end{center}
\begin{center}
$+8(m+1)\bigg(\frac{4^{m+1}-1}{3}\bigg)+6\bigg(\frac{4^{m+1}-1}{3}\bigg)\Bigg),$
\end{center}
\begin{center}
$\frac{1}{120}\Bigg(\bigg(\frac{5^{m+1}-1}{4}\bigg)^5+10\bigg(\frac{3^{m+1}-1}{2}\bigg)^3\bigg(\frac{5^{m+1}-1}{4}\bigg)+15(m+1)\bigg(\frac{5^{m+1}-1}{4}\bigg)^2$
\end{center}
\begin{center}
$+20\big(2^{m+1}-1\big)^2\bigg(\frac{5^{m+1}-1}{4}\bigg)+20\big(2^{m+1}-1\big)\bigg(\frac{3^{m+1}-1}{2}\bigg)$
\end{center}
\begin{center}$
+30(m+1)\bigg(\frac{5^{m+1}-1}{4}\bigg)+24\bigg(\frac{5^{m+1}-1}{4}\bigg)\Bigg),$
\end{center}
and
\begin{center}
$\frac{1}{720}\Bigg(\bigg(\frac{6^{m+1}-1}{5}\bigg)^6+15\bigg(\frac{4^{m+1}-1}{3}\bigg)^4\bigg(\frac{6^{m+1}-1}{5}\bigg)$
\end{center}
\begin{center}$
+45\big(2^{m+1}-1\big)^2\bigg(\frac{6^{m+1}-1}{5}\bigg)^2+15\bigg(\frac{6^{m+1}-1}{5}\bigg)^3$
\end{center}
\begin{center}
$+40\bigg(\frac{3^{m+1}-1}{2}\bigg)^3
\bigg(\frac{6^{m+1}-1}{5}\bigg)
+120(m+1)\bigg(\frac{3^{m+1}-1}{2}\bigg)\bigg(\frac{4^{m+1}-1}{3}\bigg)$
\end{center}
\begin{center}$
+40\bigg(\frac{6^{m+1}-1}{5}\bigg)^2+90\big(2^{m+1}-1\big)^2\bigg(\frac{6^{m+1}-1}{5}\bigg)$
\end{center}
\begin{center}$
+90\big(2^{m+1}-1\big)\bigg(\frac{6^{m+1}-1}{5}\bigg)
+144(m+1)\bigg(\frac{6^{m+1}-1}{5}\bigg)+120\bigg(\frac{6^{m+1}-1}{5}\bigg)\Bigg).$
\end{center}
Once again, the sequences given by these expressions appear to be little-known. The following tables give the values of (8) for $n,m\leq 6$.

\begin{table}[h]
\begin{tabular}{|l|l|l|l|l|l|l|}
\hline
  & $n=1$  & $n=2$  & $n=3$  & $n=4$  & $n=5$   \\ \hline
$m=1$ & 2 & 6 & 16 & 45 & 121 \\ \hline
$m=2$ & 3 & 28 & 390 & 8,442 & 244,910 \\ \hline
$m=3$ & 4 & 120 & 10,760 & 2,180,845 & 770,763,470 \\ \hline
$m=4$ & 5 & 496 & 295,603 & 563,483,404 & 2,421,556,983,901 \\ \hline
$m=5$ & 6 & 2,016 & 8,039,304 & 144,651,898,755 & 2,370,422,688,990,078 \\ \hline
$m=6$ & 7 & 8,128 & 217,629,416 & 37,057,640,711,850 & 23,683,244,198,577,149,289 \\ \hline

\end{tabular}
\end{table}

\begin{table}[H]
\begin{tabular}{|l|l|l|l|l|l|l|}
\hline
  & $n=6$  \\ \hline
$m=1$ & 338  \\ \hline
$m=2$ & 8,967,034 \\ \hline
$m=3$ & 419,527,164,799 \\ \hline
$m=4$ & 19,636,295,549,860,505 \\ \hline
$m=5$ & 916,720,535,022,517,503,173 \\ \hline
$m=6$ & 42,775,066,732,111,188,868,070,978  \\ \hline
\end{tabular}
\end{table}

\section{\textbf{\textit{($\chi$,$\zeta$)}}-Patterns}

In closing, we briefly place the relation $\sim$ from \S 2 into a more general context. Let $G$ be a finite group, and let $N$ and $M$ be finite nonempty sets. Suppose  that $\chi:G\to Sym(N)$ and $\zeta:G\to Sym(M)$ are group homomorphisms. Denote the set of all functions from $N$ into $M$ by $M^N$. (This notation comes from \cite{de1972enumeration}.) De Bruijn introduced the equivalence relation $E_{\chi,\zeta}$ on $M^N$ defined by 
\begin{center}
$(f_1,f_2)\in E_{\chi,\zeta} \iff f_2\circ \chi(\gamma)= \zeta(\gamma)\circ f_1$ for some $\gamma \in G$.
\end{center}
\begin{example} \normalfont \cite{de1972enumeration} Suppose that $N$ is a set of size $n\in \mathbb{N}$, and define an equivalence relation $S$ on the set of all mappings of $N$ into itself by 
\begin{center}
$(f_1,f_2)\in S \iff f_2\circ \gamma=\gamma\circ f_1 \text{ for some } \gamma\in Sym(N)$.
\end{center}
Letting $G=Sym(N)$, $M=N$, and $\chi=\zeta$ be the identity homomorphism on $Sym(N)$ shows that $S$ is a special case of the relation $E_{\chi,\zeta}$. Moreover, the sequence in Remark 2 gives the number of equivalence classes under $S$ for $n=1,2,3\ldots $. (See \S 3 of \cite{de1972enumeration}.)
\end{example}

The relation $E_{\chi,\zeta}$ stems from the left action of $G$ on $M^N$ given by
\begin{center}
$\gamma \cdot f = \zeta(\gamma) \circ f\circ \chi(\gamma^{-1}),$
\end{center}
for all $\gamma\in G$, $f\in M^N$. De Bruijn referred to the orbits of $G$ on $M^N$ as \textbf{\textit{($\chi,\zeta$)}-patterns}, and provided a formula for the number of these by applying Burnside's Lemma, and then evaluating $\left\vert\{f\in M^N : \gamma \cdot f = f\}\right\vert$ for each $\gamma \in G$. (See \cite{de1972enumeration}.) But the relation $\sim$ on the set 
$\Omega=\{\text{mappings of } D \text{ into } R\}$,
where $0< \left\vert D \right\vert < \infty$ and $R=\{W\in D^+: \left\vert W\right\vert \leq m\}$, is a special instance of the relation $E_{\chi,\zeta}$. To see this, take $N=D$, $M=R$, and 
$G=Sym(D)$. Let $\chi$ be the identity homomorphism on $Sym(D)$, and define $\zeta:G\to Sym(R)$ by
\begin{center}
$\zeta(g)=\phi_g|_R$,
\end{center}
for all $g\in Sym(D)$. Then for any $g,g'\in Sym(D)$,
\begin{center}
$\zeta(g\circ g') = \phi_{g\circ g'}|_R = (\phi_g\text{ }\circ\text{ }\phi_{g'})|_R = \phi_g|_R\text{ }\circ\text{ } \phi_{g'}|_R = \zeta(g)\circ \zeta (g')$,
\end{center}
so $\zeta$ is a group homomophism. Now, for any $f_1,f_2\in \Omega$, we have 

\begin{align*}
f_1\sim f_2 &\iff f_2\text{ }\circ\text{ } g = \phi_g \text{ }\circ\text{ } f_1 = \phi_g|_R\text{ } \circ \text{ } f_1 \text{ for some } g\in Sym(D)\\ &\iff f_2\text{ }\circ\text{ } \chi(g) = \zeta(g)\text{ } \circ \text{ }f_1 \text{ for some } g\in Sym(D) \\ &\iff (f_1,f_2)\in E_{\chi,\zeta}.\\
\end{align*}
It follows that the equivalence classes in $\Omega$ under the relation $\sim$ are $(\chi,\zeta)$-patterns, for $\chi$,$\zeta$ chosen as above. In particular, our Theorem 1 is a special case of de Bruijn's formula.

\end{document}